\documentclass[10pt]{amsart}

\usepackage{amssymb,latexsym,amsmath,amsthm,amscd}
\usepackage[left=3cm,top=3cm,right=3cm,bottom=3cm]{geometry}
\usepackage{graphicx}

\providecommand{\abs}[1]{\left\vert #1 \right\vert}

\theoremstyle{plain}
\newtheorem{thm}{Theorem}
\newtheorem*{thm*}{Theorem}
\newtheorem{lemma}[thm]{Lemma}

\newtheorem*{lem*}{Lemma}

\newtheorem*{prop*}{Proposition}

\newtheorem*{cor*}{Corollary}

\newtheorem*{conj*}{Conjecture}

\theoremstyle{definition}

\newtheorem*{cons*}{Construction}

\newtheorem*{df*}{Definition}

\newtheorem*{nota*}{Notation}
\newtheorem*{problem*}{Problem}

\newtheorem*{qu*}{Question}

\newtheorem{rmk}[thm]{Remark}
\newtheorem*{rmk*}{Remark}

\newtheorem*{ex*}{Example}

\begin{document}

\title{Lower Bounds for the Least Prime in Chebotarev}
\author{Andrew Fiori}

\thanks{The author thanks the University of Lethbridge for providing a stimulating environment to conduct this work and in particular Nathan Ng and Habiba Kadiri for directing him towards this project.
The author would also like to thank the referee whose suggestions simplified the proof of our main result.}

\ifx\XMetaElseArt\undefined
\email{andrew.fiori@uleth.ca}
\else
\ead{andrew.fiori@uleth.ca}
\fi
\address{%
Mathematics and Computer Science,
4401 University Drive,
University of Lethbridge,
Lethbridge, Alberta, T1K 3M4}

\begin{abstract}
In this paper we show there exists an infinite family of number fields $L$, Galois over $\mathbb{Q}$, for which the smallest prime $p$ of $\mathbb{Q}$ which splits completely in $L$
 has size at least $\left( \log(\abs{D_L}) \right)^{2+o(1)}$. This gives a converse to various upper bounds, which shows that they are best possible. 
\end{abstract}
\maketitle

\section{Introduction}

The purpose of this note is to prove the following result.
\begin{thm}\label{thm:1}
 There exists an infinite family of number fields $L$, Galois over $\mathbb{Q}$, for which the smallest prime $p$ of $\mathbb{Q}$ which splits completely in $L$
 has size at least
 \[ (1+o(1))\left(\frac{3e^{\gamma}}{2\pi}\right)^2\left(\frac{\log(\abs{D_L})\log(2\log\log(\abs{D_L}))}{\log\log(\abs{D_L})}\right)^2 \]
 as the absolute discriminant $D_L$ of $L$ over $\mathbb{Q}$, tends to infinity.
\end{thm}
The result is independent of the Generalized Riemann Hypothesis. 
The result complements the existing literature on what is essentially a converse problem, stated generally as
\begin{problem*}
 Let $K$ be a number field, and $L$ be a Galois extension of $K$, for any conjugacy class $\mathcal{C}$ in $\Gamma(L/K)$, the Galois group of $L/K$, show that the smallest (in norm) 
 unramified degree one prime $\mathfrak{p}$ of $K$ for which the conjugacy class ${\rm Frob}_{\mathfrak{p}}$ is $\mathcal{C}$ is {\it small} relative to $\abs{D_L}$, the absolute discriminant of $L/K$.
\end{problem*}
Solutions to this problem have important applications in the explicit computation of class groups (see \cite{BelabasSmallGeneratorsClassGroup}) where smaller is better.
Some of the history of just how small we can get is summarized below:
\begin{itemize}
 \item Lagarias and Odlyzko showed $N_{K/\mathbb{Q}}(\mathfrak{p}) < \left( \log(\abs{D_L}) \right)^{2+o(1)}$ conditionally on GRH (see \cite{LagariasOdlyzko}).
 \item Bach and Sorenson gave an explicit constant $C$ so that $N_{K/\mathbb{Q}}(\mathfrak{p}) < C  \left( \log(\abs{D_L}) \right)^{2}$ conditionally on GRH (see \cite{BachSorenson}).
 \item Lagarias, Montgomery,
and Odlyzko showed there is a constant $A$ such that $N_{K/\mathbb{Q}}(\mathfrak{p}) < \abs{D_L}^A$ (see \cite{LagariasMontgomeryOdlyzko}).
 \item Zaman showed $N_{K/\mathbb{Q}}(\mathfrak{p}) < \abs{D_L}^{40}$ for $D_L$ sufficiently large (see \cite{Zaman}).
 \item Kadiri, Ng and Wong improved this to $N_{K/\mathbb{Q}}(\mathfrak{p}) < \abs{D_L}^{16}$ for $D_L$ sufficiently large (see \cite{KadiriNgWong}).
 \item Ahn and Kwon showed $N_{K/\mathbb{Q}}(\mathfrak{p}) < \abs{D_L}^{12577}$ for all $L$ (see \cite{AhnKwon}).
\end{itemize}
By the above, Theorem \ref{thm:1} and the GRH bound above are best possible up to the exact $o(1)$ term.

\begin{rmk*}
The family under consideration will be a subfamily of the Hilbert class fields of quadratic imaginary extensions of $\mathbb{Q}$.
All of the Galois groups will be generalized dihedral groups, and in the family the degree of the extensions goes to infinity.

We also would like to point out the work of Sandari (see \cite[Sec. 1.3]{SardariLeastPrime}) where some similar features of this family are remarked on in a different context.
\end{rmk*}

\section{Proofs}

We first recall a few basic facts from algebraic number theory and class field theory.
\begin{lemma}
Let $K = \mathbb{Q}(\sqrt{-d})$ where $d=\abs{{\rm disc}(K)}$, let $\mathfrak{p}$ be a principal prime ideal of $K$.
If we have $N_{K/\mathbb{Q}}(\mathfrak{p}) = (p)$ then $p$ is a norm of $\mathcal{O}_K$ and hence $p\ge d/4$.
\end{lemma}
\begin{proof}
Assuming $\mathfrak{p}$ is principally generated by $x$, then $N_{K/\mathbb{Q}}(\mathfrak{p})$ is principally generated by $N_{K/\mathbb{Q}}(x)$. As norms from $K$ are positive, this gives that $p$ must be a norm.

We next note that for $x+y\sqrt{-d}\in \mathcal{O}_K$ the expression $N_{K/\mathbb{Q}}(x+y\sqrt{-d}) = x^2+dy^2$ cannot be prime if $y=0$.
Now, because $\mathcal{O}_k \subset \tfrac{1}{2}\mathbb{Z} + \tfrac{\sqrt{-d}}{2}\mathbb{Z}$ we conclude that if the norm is a prime, then $y\ge \frac{1}{2}$, from which it follows that if
$p$ is a norm then $p\ge d/4$.
\end{proof}

\begin{lemma}\label{lem:sizep}
Let $K = \mathbb{Q}(\sqrt{-d})$ where $d=\abs{{\rm disc}(K)}$, suppose that $H$ is the Hilbert class field of $K$.
If $p$ is a prime of $\mathbb{Z}$ which splits completely in $H$, then $p$ splits in $K$ as $(p) = \mathfrak{p}_1\mathfrak{p}_1$ where both $\mathfrak{p}_1$ and $\mathfrak{p}_2$ are principal and $N_{K/\mathbb{Q}}(\mathfrak{p}_i) = (p)$.
In particular, by the previous lemma $p\ge d/4$.
\end{lemma}
\begin{proof}
 The first claim is clear because ramification degrees, inertia degrees and hence splitting degrees are multiplicative in towers.
 That $\mathfrak{p}_i$ must be principal is a consequence of class field theory.
 Principal ideals for $\mathcal{O}_K$ map to the trivial Galois element for the Galois group of the Hilbert class field.
 However, for unramified prime ideals this map gives Frobenius. As the Frobenius element is trivial precisely when the inertial degree is $1$, equivalently for Galois fields when the prime splits completely, we conclude the result.
\end{proof}

\begin{rmk}
 Denote by $\chi_d$ the quadratic Dirichlet character with fundamental discriminant $-d$.
 The main idea of the proof is to use the class number formula with lower bounds for $L(1,\chi_d)$. Using Siegel's ineffective bound gives
 \[ d = h_K^{2+o(1)} = \log\abs{D_H}^{2+o(1)}. \]
 To obtain our precise result we refine the $o(1)$ term using extreme values of $L(1,\chi_d)$.
\end{rmk}

\begin{lemma}\label{lem:disc}
Let $K = \mathbb{Q}(\sqrt{-d})$ where $d=\abs{{\rm disc}(K)}>16$, suppose that $H$ is the Hilbert class field of $K$.
Then 
\[ \log\abs{D_H} = h_K \log(d) =  \frac{1}{\pi}L(1,\chi_d)\sqrt{d}\log(d) \]
where $h_K$ is the class number of $K$, $D_H$ is the discriminant of $H$ and $\chi_d$ is the quadratic Dirichlet character with fundamental discriminant $-d$.
\end{lemma}
\begin{proof}
The first equality is immediate from the multiplicativity of the discriminant in towers, the second follows from the analytic class number formula
\[ h_{K} = \frac{\sqrt{d}}{\pi}L(1,\chi_d).\qedhere \]
\end{proof}

The estimates on the extreme values of $L(1,\chi_d)$ which we need are the following. 

\begin{thm}\label{thm:hkub}
There exists a family of quadratic imaginary fields $K = \mathbb{Q}(\sqrt{-d})$ where $d=\abs{{\rm disc}(K)}$ such that for $\chi_d$, the quadratic Dirichlet character with fundamental discriminant $-d$, we have
\[ L(1,\chi_d) < (1+o(1))\frac{\pi^2}{6e^{\gamma}\log\log (d)}. \]
\end{thm}
A result of this sort was originally proven by Littlewood conditional on the generalized Riemann hypothesis (see \cite{LittlewoodClassNumber}), his result was proven unconditionally by Paley (see \cite{PaleyCharacters})
the version stated here follows from the work of Chowla (see \cite{Chowla}). 
It is possible that the work of Granville and Soundararajan (see \cite{granvillesound}) can further refine the constants in the above, and consequently those in Theorem \ref{thm:1}.

The following proof includes several significant simplifications suggested by the referee. We would like to thank them for these valuable suggestions.
\begin{proof}[Proof of Theorem \ref{thm:1}]
 We consider the family of fields $L = H_K$ where $K$ is a field from the infinite family of Theorem \ref{thm:hkub} for which $d > 16$.
 To complete the proof we introduce some notation, define
 \[ x_d = L(1,\chi_d)\log\log(d) \qquad\text{and}\qquad f_d(x) = \frac{x\sqrt{d}\log(d)}{\pi\log\log(d)}. \]
 Then by our choice of $d$ we have
 \[ x_d < \frac{\pi^2}{6e^{\gamma}} + o(1)\]
 and by Lemma \ref{lem:disc} we have
 \[ \log{\abs{D_L}} = f_d(x_d). \]
 Now because the function $y \mapsto \frac{y\log(2\log(y))}{\log(y)}$ is increasing for $y>e$ and as 
 \[ f_d(x_d) = \log{\abs{D_L}} = h_{K}\log(d) \ge \log(16) > e \]
 it follows that
 \begin{align*} \frac{ \log\abs{D_L}\log(2\log\log\abs{D_L})}{\log\log\abs{D_L}} 
    &= \frac{f_d(x_d)\log(2\log (f_d(x_d)))}{\log(f_d(x_d))} \\
    &\leq  \frac{f_d(\frac{\pi^2}{6e^\gamma} + o(1))\log(2\log(f_d(\frac{\pi^2}{6e^\gamma} + o(1))))}{\log(f_d(\frac{\pi^2}{6e^\gamma} + o(1)))}\\
    &\leq (1+o(1))\frac{\pi}{3e^\gamma}\sqrt{d}.
 \end{align*}
 Combining the above with the bounds $p\ge \frac{d}{4}$ from Lemma \ref{lem:sizep} we obtain the result. 
 \end{proof}

\section{Numerics}

Table \ref{table:1} illustrates the phenomenon by giving the ratio
\[ {\rm Ratio} =  p /  \left(\frac{3e^{\gamma}}{2\pi}\right)^2\left(\frac{\log(\abs{D_L})\log(2\log\log(\abs{D_L}))}{\log\log(\abs{D_L})}\right)^2 \]
for an example of a the Hilbert class field of a quadratic imaginary field of each class number less than $100$ with large discriminant.

Note that in Table \ref{table:1} we have $K=\mathbb{Q}(\sqrt{-d})$ and $\abs{D_L} = d^{h_K}$.

\begin{table}[ht]
 \caption{Examples of smallest split primes in Hilbert class fields of $\mathbb{Q}(\sqrt{-d})$}
\vspace{-3pt}
{\small\label{table:1}
\renewcommand{\arraystretch}{0.99}
\begin{tabular}{| c | c | c| c |}
\hline
$h_K$ & $d$ & $p$ & Ratio \\
\hline
$ 1        $ & $ 163     $  & $     41 $ & $ 4.1557$ \\
 $ 2        $ & $ 427     $  & $    107 $ & $ 2.4287$ \\
 $ 3        $ & $ 907     $  & $    227 $ & $ 2.1188$ \\
 $ 4        $ & $ 1555    $  & $    389 $ & $ 1.9476$ \\
 $ 5        $ & $ 2683    $  & $    673 $ & $ 2.0276$ \\
 $ 6        $ & $ 3763    $  & $    941 $ & $ 1.9222$ \\
 $ 7        $ & $ 5923    $  & $   1481 $ & $ 2.1071$ \\
 $ 8        $ & $ 6307    $  & $   1579 $ & $ 1.7569$ \\
 $ 9        $ & $ 10627   $  & $   2657 $ & $ 2.1729$ \\
 $ 10       $ & $ 13843   $  & $   3461 $ & $ 2.2386$ \\
 $ 11       $ & $ 15667   $  & $   3917 $ & $ 2.0939$ \\
 $ 12       $ & $ 17803   $  & $   4451 $ & $ 1.9938$ \\
 $ 13       $ & $ 20563   $  & $   5147 $ & $ 1.9503$ \\
 $ 14       $ & $ 30067   $  & $   7517 $ & $ 2.3373$ \\
 $ 15       $ & $ 34483   $  & $   8623 $ & $ 2.3173$ \\
 $ 16       $ & $ 31243   $  & $   7817 $ & $ 1.9050$ \\
 $ 17       $ & $ 37123   $  & $   9281 $ & $ 1.9719$ \\
 $ 18       $ & $ 48427   $  & $  12107 $ & $ 2.2225$ \\
 $ 19       $ & $ 38707   $  & $   9677 $ & $ 1.6747$ \\
 $ 20       $ & $ 58507   $  & $  14627 $ & $ 2.1572$ \\
 $ 21       $ & $ 61483   $  & $  15373 $ & $ 2.0614$ \\
 $ 22       $ & $ 85507   $  & $  21377 $ & $ 2.5024$ \\
 $ 23       $ & $ 90787   $  & $  22697 $ & $ 2.4308$ \\
 $ 24       $ & $ 111763  $  & $  27941 $ & $ 2.6847$ \\
 $ 25       $ & $ 93307   $  & $  23327 $ & $ 2.1425$ \\
 $ 26       $ & $ 103027  $  & $  25759 $ & $ 2.1714$ \\
 $ 27       $ & $ 103387  $  & $  25847 $ & $ 2.0351$ \\
 $ 28       $ & $ 126043  $  & $  31511 $ & $ 2.2543$ \\
 $ 29       $ & $ 166147  $  & $  41539 $ & $ 2.6760$ \\
 $ 30       $ & $ 134467  $  & $  33617 $ & $ 2.1037$ \\
 $ 31       $ & $ 133387  $  & $  33347 $ & $ 1.9698$ \\
 $ 32       $ & $ 164803  $  & $  41201 $ & $ 2.2263$ \\
 $ 33       $ & $ 222643  $  & $  55661 $ & $ 2.7216$ \\
 \hline
 \end{tabular}\hfil
\begin{tabular}{| c | c | c| c |}
\hline
$h_K$ & $d$ & $p$ & Ratio \\
\hline
 $ 34       $ & $ 189883  $  & $  47491 $ & $ 2.2528$ \\
 $ 35       $ & $ 210907  $  & $  52727 $ & $ 2.3373$ \\
 $ 36       $ & $ 217627  $  & $  54409 $ & $ 2.2819$ \\
 $ 37       $ & $ 158923  $  & $  39733 $ & $ 1.6620$ \\
 $ 38       $ & $ 289963  $  & $  72493 $ & $ 2.6454$ \\
 $ 39       $ & $ 253507  $  & $  63377 $ & $ 2.2500$ \\
 $ 40       $ & $ 260947  $  & $  65239 $ & $ 2.2034$ \\
 $ 41       $ & $ 296587  $  & $  74149 $ & $ 2.3513$ \\
 $ 42       $ & $ 280267  $  & $  70067 $ & $ 2.1445$ \\
 $ 43       $ & $ 300787  $  & $  75209 $ & $ 2.1838$ \\
 $ 44       $ & $ 319867  $  & $  79967 $ & $ 2.2079$ \\
 $ 45       $ & $ 308323  $  & $  77081 $ & $ 2.0542$ \\
 $ 46       $ & $ 462883  $  & $ 115727 $ & $ 2.7990$ \\
 $ 47       $ & $ 375523  $  & $  93887 $ & $ 2.2489$ \\
 $ 48       $ & $ 335203  $  & $  83813 $ & $ 1.9638$ \\
 $ 49       $ & $ 393187  $  & $  98297 $ & $ 2.1693$ \\
 $ 50       $ & $ 389467  $  & $  97367 $ & $ 2.0743$ \\
 $ 51       $ & $ 546067  $  & $ 136519 $ & $ 2.6772$ \\
 $ 52       $ & $ 439147  $  & $ 109789 $ & $ 2.1422$ \\
 $ 53       $ & $ 425107  $  & $ 106277 $ & $ 2.0124$ \\
 $ 54       $ & $ 532123  $  & $ 133033 $ & $ 2.3604$ \\
 $ 55       $ & $ 452083  $  & $ 113021 $ & $ 1.9839$ \\
 $ 56       $ & $ 494323  $  & $ 123581 $ & $ 2.0737$ \\
 $ 57       $ & $ 615883  $  & $ 153991 $ & $ 2.4279$ \\
 $ 58       $ & $ 586987  $  & $ 146749 $ & $ 2.2565$ \\
 $ 59       $ & $ 474307  $  & $ 118583 $ & $ 1.8204$ \\
 $ 60       $ & $ 662803  $  & $ 165701 $ & $ 2.3566$ \\
 $ 61       $ & $ 606643  $  & $ 151667 $ & $ 2.1185$ \\
 $ 62       $ & $ 647707  $  & $ 161947 $ & $ 2.1768$ \\
 $ 63       $ & $ 991027  $  & $ 247759 $ & $ 3.0559$ \\
 $ 64       $ & $ 693067  $  & $ 173267 $ & $ 2.1783$ \\
 $ 65       $ & $ 703123  $  & $ 175781 $ & $ 2.1443$ \\
 $ 66       $ & $ 958483  $  & $ 239623 $ & $ 2.7278$ \\
 \hline
 \end{tabular}\hfil
\begin{tabular}{| c | c | c| c |}
\hline
$h_K$ & $d$ & $p$ & Ratio \\
\hline
 $ 67       $ & $ 652723  $  & $ 163181 $ & $ 1.9030$ \\
 $ 68       $ & $ 819163  $  & $ 204791 $ & $ 2.2546$ \\
 $ 69       $ & $ 888427  $  & $ 222107 $ & $ 2.3556$ \\
 $ 70       $ & $ 811507  $  & $ 202877 $ & $ 2.1215$ \\
 $ 71       $ & $ 909547  $  & $ 227387 $ & $ 2.2823$ \\
 $ 72       $ & $ 947923  $  & $ 236981 $ & $ 2.3061$ \\
 $ 73       $ & $ 886867  $  & $ 221717 $ & $ 2.1227$ \\
 $ 74       $ & $ 951043  $  & $ 237763 $ & $ 2.2001$ \\
 $ 75       $ & $ 916507  $  & $ 229127 $ & $ 2.0792$ \\
 $ 76       $ & $ 1086187 $  & $ 271549 $ & $ 2.3521$ \\
 $ 77       $ & $ 1242763 $  & $ 310693 $ & $ 2.5821$ \\
 $ 78       $ & $ 1004347 $  & $ 251087 $ & $ 2.0958$ \\
 $ 79       $ & $ 1333963 $  & $ 333491 $ & $ 2.6208$ \\
 $ 80       $ & $ 1165483 $  & $ 291371 $ & $ 2.2775$ \\
 $ 81       $ & $ 1030723 $  & $ 257687 $ & $ 2.0011$ \\
 $ 82       $ & $ 1446547 $  & $ 361637 $ & $ 2.6277$ \\
 $ 83       $ & $ 1074907 $  & $ 268729 $ & $ 1.9851$ \\
 $ 84       $ & $ 1225387 $  & $ 306347 $ & $ 2.1765$ \\
 $ 85       $ & $ 1285747 $  & $ 321443 $ & $ 2.2210$ \\
 $ 86       $ & $ 1534723 $  & $ 383681 $ & $ 2.5366$ \\
 $ 87       $ & $ 1261747 $  & $ 315437 $ & $ 2.0941$ \\
 $ 88       $ & $ 1265587 $  & $ 316403 $ & $ 2.0564$ \\
 $ 89       $ & $ 1429387 $  & $ 357347 $ & $ 2.2395$ \\
 $ 90       $ & $ 1548523 $  & $ 387137 $ & $ 2.3529$ \\
 $ 91       $ & $ 1391083 $  & $ 347771 $ & $ 2.1002$ \\
 $ 92       $ & $ 1452067 $  & $ 363017 $ & $ 2.1371$ \\
 $ 93       $ & $ 1475203 $  & $ 368801 $ & $ 2.1244$ \\
 $ 94       $ & $ 1587763 $  & $ 396943 $ & $ 2.2212$ \\
 $ 95       $ & $ 1659067 $  & $ 414767 $ & $ 2.2638$ \\
 $ 96       $ & $ 1684027 $  & $ 421009 $ & $ 2.2501$ \\
 $ 97       $ & $ 1842523 $  & $ 460633 $ & $ 2.3882$ \\
 $ 98       $ & $ 2383747 $  & $ 595939 $ & $ 2.9359$ \\
 $ 99       $ & $ 1480627 $  & $ 370159 $ & $ 1.9012$ \\
\hline
\end{tabular}
}
\end{table}


\renewcommand{\MR}[1]{}
\providecommand{\bysame}{\leavevmode\hbox to3em{\hrulefill}\thinspace}
\providecommand{\MRhref}[2]{%
  \href{http://www.ams.org/mathscinet-getitem?mr=#1}{#2}
}
\providecommand{\href}[2]{#2}

\end{document}